\documentclass[11pt,a4paper,reqno]{amsart}

\usepackage{amsmath,amssymb,enumitem, verbatim,stmaryrd,xcolor,microtype,graphicx,aliascnt,fancyvrb,mathtools}
\usepackage{mathptmx}
\usepackage{tikz-cd}
\usepackage{tikz-3dplot}

\usepackage[T1]{fontenc}
\usepackage[utf8]{inputenc}
\usepackage[normalem]{ulem}
\usepackage[english]{babel} 
\usepackage[top=3.5cm,bottom=3.5cm,left=3.2cm,right=3.2cm]{geometry}
\usepackage[bookmarksdepth=2,linktoc=page,colorlinks,linkcolor={blue!80!black},citecolor={red!80!black},urlcolor={blue!80!black},pdftitle={Title},pdfauthor={author}]{hyperref}

\usepackage{tikz}\usetikzlibrary{matrix,arrows,decorations.markings}
\usepackage{tikz-cd}
\usepackage{enumitem}
\usepackage[dvipsnames]{xcolor}

\allowdisplaybreaks

\usepackage[mathscr]{euscript}           
\usepackage{mathptmx}  
\usepackage[skip=4pt plus1pt, indent=13pt]{parskip}
\usepackage{etoolbox}\makeatletter\patchcmd{\@startsection}{\@afterindenttrue}{\@afterindentfalse}{}{}\makeatother   
\patchcmd{\section}{\scshape}{\bfseries}{}{}\makeatletter\renewcommand{\@secnumfont}{\bfseries}\makeatother 
\usepackage[backgroundcolor=orange!30!white,linecolor=orange!80!white,textsize=footnotesize]{todonotes} \makeatletter \providecommand \@dotsep{5} \def\listtodoname{List of Todos} \def\listoftodos{\@starttoc{tdo}\listtodoname} \makeatother 

\addto\extrasenglish{         }
\theoremstyle{plain}
\newtheorem{thm}{Theorem}[section] 
\newaliascnt{lemma}{thm}\newtheorem{lemma}[lemma]{Lemma}\aliascntresetthe{lemma}
\newaliascnt{cor}{thm}\aliascntresetthe{cor}
\newaliascnt{prop}{thm}\newtheorem{prop}[prop]{Proposition}\aliascntresetthe{prop}
 
\newaliascnt{propA}{thmA}\aliascntresetthe{propA}

\newtheorem*{thm*}{Theorem}
\newtheorem*{lem*}{Lemma}
\newtheorem*{cor*}{Corollary}
\newtheorem*{problem*}{Problem}

\theoremstyle{definition}
\newaliascnt{definition}{thm}\newtheorem{definition}[definition]{Definition}\aliascntresetthe{definition}
\newaliascnt{rem}{thm}\newtheorem{rem}[rem]{Remark}\aliascntresetthe{rem}
\newaliascnt{question}{thm}\aliascntresetthe{question}
\newaliascnt{example}{thm}\newtheorem{example}[example]{Example}\aliascntresetthe{example}

\newtheorem*{df*}{Definition}
\newtheorem*{ex*}{Example}
\newtheorem*{rem*}{Remark}
\usepackage{etoolbox}
\makeatletter
\patchcmd{\@startsection}{\@afterindenttrue}{\@afterindentfalse}{}{}            
\patchcmd{\part}{\bfseries}{\bfseries\LARGE}{}{}
\patchcmd{\section}{\scshape}{\bfseries}{}{}\renewcommand{\@secnumfont}{\bfseries} 
\patchcmd{\@settitle}{\uppercasenonmath\@title}{\large}{}{}
\patchcmd{\@setauthors}{\MakeUppercase}{}{}{}
  
\addto{\captionsenglish}{} 
\addto{\captionsenglish}{} 
\makeatother

\usepackage{fancyhdr}

\DeclareRobustCommand{\gobblefour}[5]{}    

\DeclareFontFamily{OT1}{pzc}{}    
\DeclareFontShape{OT1}{pzc}{m}{it}{<-> s * [1.10] pzcmi7t}{}
\DeclareMathAlphabet{\mathpzc}{OT1}{pzc}{m}{it}
\DeclareSymbolFont{sfoperators}{OT1}{bch}{m}{n} \DeclareSymbolFontAlphabet{\mathsf}{sfoperators} \makeatletter\def\operator@font{\mathgroup\symsfoperators}\makeatother
\DeclareSymbolFont{cmletters}{OML}{cmm}{m}{it}              
\DeclareSymbolFont{cmsymbols}{OMS}{cmsy}{m}{n}
\DeclareSymbolFont{cmlargesymbols}{OMX}{cmex}{m}{n}
\DeclareMathSymbol{\myjmath}{\mathord}{cmletters}{"7C}     \let\jmath\myjmath

\DeclareMathOperator{\im}{im}

\newcommand\BC{{\mathbb C}}

\newcommand\BLP{\mathbb{LP}}
\newcommand\BK{{\mathbf K}}

\newcommand\Bn{{\mathbf n}}
\newcommand\BP{{\mathbb P}}

\newcommand\BR{{\mathbb R}}
\NewCommandCopy{\Sec}{\S}
\renewcommand\S{{\mathbb S}}

\newcommand\BZ{{\mathbb Z}}

\newcommand\bF{{\mathbf F}}
\newcommand\bI{{\mathbf I}}

\newcommand\bP{{\mathbf P}}
\newcommand\bR{{\mathbf R}}

\newcommand\bk{{\mathbf{k}}}

\newcommand\cM{{\mathcal M}}

\newcommand\cO{{\mathcal O}}

\renewcommand\int{\textup{int}}

\renewcommand\geq{\geqslant}
\renewcommand\leq{\leqslant}

\title{On the Hilbert polynomial of the linked projective space}

\author[]{FELIPE DE LEÓN}
\address{Universidad de Santiago de Chile (USACH), Avenida Libertador Bernardo O'Higgins no. 3363, Estaci\'{o}n Central, Santiago, Chile}
\email{felipe.saenz@usach.cl}

\author[]{EDUARDO ESTEVES}
\address{Instituto de Matemática Pura e Aplicada (IMPA), Estrada Dona Castorina 110, 22460-320 Rio de Janeiro RJ, Brazil}
\email{esteves@impa.br}

\author[]{EDUARDO VITAL}
\address{Universit\"at Bielefeld, 
Universit\"atsstra{\ss}e 25, 33615 Bielefeld, Germany}
\email{evital@math.uni-bielefeld.de}

\begin{document}

\begin{abstract}

Linked projective spaces are quiver Grassmannians of subspaces of dimension 1 of certain quiver representations. Degenerations of linear series produce these representations, with the limit divisors parameterized by the associated linked projective spaces. It is not known whether all linked projective spaces arise this way. If they do, they are degenerations of the (small) diagonal in a product of projective spaces. In any case, we prove here that they have the (multivariate) Hilbert polynomial of the diagonal. To achieve this, we first extend a Hilbert-polynomial formula for multiplicity-free varieties to (simple) normal-crossings schemes with multiplicity-free strata in products of projective spaces. Then we prove that a linked projective space is normal-crossings, by describing it locally in terms of Mustafin varieties. Finally, we use a relation between intersections of components of the linked projective space and certain polytopes in the tiling of a simplex associated to the linked net to prove that our formula for the Hilbert polynomial applies.

\end{abstract}

\maketitle

\section{Introduction}\label{sec:introduction}

Linked projective spaces arise naturally in the study of degenerations of linear series. The classical theory of limit linear series provides a way to retain partial linear-series data when a smooth curve specializes to a reducible one \cite{Eisenbud1986}. Subsequent constructions make clear that the full linear-series data is necessary to describe the limiting divisors themselves \cite{EstevesOsserman2013}.  In the framework developed by Santos and the second-named and third-named authors, building on ideas by Osserman \cite{Osserman2019}, applied to specializations of varieties of any dimension, full data are encoded by certain quiver representations, called linked nets \cite{Esteves22072025,Esteves2025}. The quiver Grassmannian parameterizing one-dimensional subrepresentations of a linked net is called its \emph{linked projective space}.  When a linked net comes from a degenerating family of linear series, its associated linked projective space parameterizes the schematic limits of divisors in the family \cite{Esteves2025}. Thus linked projective spaces provide a bridge between the geometry of degenerations and the representation theory of quivers.

There is a second bridge, this time to non-Archimedean and tropical
geometry.  Mustafin varieties are degenerations of projective space attached
to finite configurations of lattices in a Bruhat--Tits building.  Their
geometry reflects both the convexity of the lattice configuration and the
combinatorics of tropical convexity
\cite{Faltings2001,Cartwright2011}.  Quiver-theoretic descriptions of
degenerations associated to lattice configurations were developed further
by Hahn and Li \cite{HahnLi2021}, and He and Zhang \cite{10.1093/imrn/rnab241,He_Zhang_2025}. One of our first observations
is that these two classes of objects meet more closely than might initially
be expected: a linked projective space generated by a polygon is a Mustafin
variety associated to an explicit convex lattice configuration, and an
arbitrary linked projective space is locally of this form.  This gives local
models for linked projective spaces and, in particular, shows that they are
simple normal crossings, our \autoref{thm:linked-projective-space-snc}.

The intersection theory of linked projective spaces has recently revealed
a third connection, with polymatroids.  To a linked net one can associate a
collection of base polytopes which tile a simplex; the faces of this tiling
are indexed by polygons in the generating set
\cite{Chow.class}. 
This description was used in loc.~cit.~to prove that a linked projective space has the Chow
class of the small diagonal.  The result is striking
because it says that, at the level of multidegrees, a generally reducible
linked projective space behaves like a single projective space embedded
diagonally in a product.  It leaves open, however, whether every linked
projective space is actually a flat degeneration of that diagonal.

In this paper we use the tiling again to strengthen the equality of Chow classes to an equality of multivariate Hilbert polynomials. To do this, we use recent work by Castillo et alii~\cite{castillo2025kpolynomialsmultiplicityfreevarieties}, to give a short proof of a formula for the Hilbert polynomial of multiplicity-free (irreducible) varieties in a product of projective spaces as a sum over lattice points of the independence polytope of the variety. (An equivalent expression appears in work by Eur and Larson~\cite{Eur_Larson_2026}, while a related formula in the setting of multiview varieties appears in Conca and Tsakiris~\cite{Conca_Tsakiris_2022}.) We extend this result to  simple normal-crossings schemes whose nonempty strata are multiplicity-free. 

We describe now our main result in more detail. Let $\mathfrak V$ be a
nontrivial exact finitely generated linked net of vector spaces of dimension
$r+1$ over a $\BZ^n$-quiver, and let $H$ be its minimum set of generators; see \autoref{sec:quivers-mustafin-linked-nets}.
The linked projective space
\[
    \BLP(\mathfrak V)\subseteq\BP_H\coloneqq\prod_{u\in H}\BP(V_u)
\]
inherits the natural $H$-grading from the product.  If
$\Bn=(n_u)_{u\in H}$, our main theorem states that
\[
    P_{\BLP(\mathfrak V)}(\Bn)
    =\binom{r+\sum_{u\in H}n_u}{r}.
\]
The right-hand side is precisely the multivariate Hilbert polynomial of the
small diagonal
\[
    \begin{tikzcd}
        \BP^r \arrow[r, hook] & (\BP^r)^H.
    \end{tikzcd}
\]
In particular, the Hilbert polynomial is
independent of the quiver and of the individual maps in $\mathfrak V$.
This is \autoref{thm:multivariate-hilbert-polynomial}.

We briefly describe the mechanism behind its proof. First, we express the Hilbert polynomial of a multiplicity-free subvariety $Y\subseteq\BP_H$ as a sum indexed by the lattice points of the independence polytope $D_{Y}$ of its polymatroidal rank function $\varphi_Y$, our  \autoref{prop:multiplicity-free-hilbert-polynomial}:
\begin{equation}\label{eq:poly-formula}
    P_{Y}(\Bn) =\sum_{a\in D_Y}
    \prod_{u\in H}\binom{n_u+a_u-1}{a_u}.
\end{equation}

The subvariety of interest to us is the scheme-theoretic image $Y_W$ of the rational map
\[
    \begin{tikzcd}
        \BP(W) \arrow[r, dashed] & \displaystyle\prod_{u\in H}\BP(V_u)
    \end{tikzcd}
\] 
associated to a vector subspace $W\subseteq V_H\coloneqq\bigoplus_{u\in H}V_u$. In this case, we prove that $\varphi_{Y_W}$ is the Dilworth truncation of $\rho_W-1$, where $\rho_W$ is the polymatroidal rank function of $W$, our \autoref{prop:projection-rank-dilworth-truncation}, and use this to give a geometric characterization of $D_W\coloneqq D_{Y_W}$ in terms of the base polytope $\bP_W$ of $\rho_W$, our \autoref{prop:interior-criterion}.  Images of rational maps of projective spaces are multiplicity-free by \cite{Li_2018}, so we apply the formula above, which yields \autoref{prop:general-vector-space-hilbert-polynomial}.

We apply our formula to the strata of a linked projective space $X\coloneqq\BLP(\mathfrak V)$.  As shown in \cite{Chow.class,Esteves2025}, the irreducible components of $X$ are indexed by the vertices of $H$, while their nonempty intersections $X_\Delta$ are indexed by polygons $\Delta\subseteq H$.  Furthermore, $X_\Delta$ is the scheme-theoretic image of the rational map given by a certain vector space $W_\Delta\subseteq V_H$ associated to $\Delta$. Hence $X_\Delta$ is multiplicity-free, and thus  
\[
P_{X_\Delta}(\Bn)
=\sum_{a\in D_\Delta}
\prod_{u\in H}\binom{n_u+a_u-1}{a_u},
\]
where $D_{\Delta}$ is the independence polytope of $X_\Delta$. 

The local Mustafin description implies that $X$ has simple normal crossings, so its Hilbert polynomial is obtained from those of the strata by inclusion--exclusion.  At first, this produces an alternating sum over all polygons, which can then be written by \autoref{prop:normal-crossings-euler-cancellation} as
\[
P_X(\Bn)=\sum_{a\in\BZ_{\geq0}^H}\chi(K_a)\prod_{u\in H}\binom{n_u+a_u-1}{a_u},
\]
where $\chi(K_a)$ is the Euler characteristic of the simplicial complex $K_a$ formed by all polygons $\Delta$ with $a\in D_{\Delta}$. 
The polymatroidal tiling in \cite{Chow.class} turns that Euler characteristic into a topological one: It follows from the geometric characterization of the $D_\Delta$ that 
for each lattice point $a$, either $K_a$ is trivial, in which case $|a|>r$, or $K_a$ is the nerve of a cover of a simplex by closed convex subsets, in which case Nerve Theorem \cite{BauerKerberRollRolle_2023} yields that $K_a$ is homotopically equivalent to the simplex, whence contractible. Thus $\chi(K_a)=1$ for each $a$ with $|a|\leq r$. 

Then 
another application of the tiling in \cite{Chow.class}, 
followed by Multivariate Chu--Vandermonde Identity \cite{Vignat_Moll_2015} and the 
hockey-stick identity reduces the
resulting sum to the single binomial coefficient in \autoref{thm:multivariate-hilbert-polynomial}.

Our argument also explains why the conclusion is stronger than equality
of Chow classes.  The Chow class records the terms of top total degree in
the multivariate Hilbert polynomial, whereas the formula above controls all
of its coefficients.  On the other hand, equality of Hilbert polynomials is
still weaker than constructing a flat family connecting
$\BLP(\mathfrak V)$ to the diagonal; cf.~\cite{CS2010}.  The latter remains the most immediate
open problem suggested by the theory:
\begin{quote}
    Is every linked projective space a flat degeneration of the small diagonal?
\end{quote}
A positive answer would give a geometric explanation for both the Chow-class
calculation and the Hilbert-polynomial formula.  It would also be useful to
determine whether such a degeneration can be chosen compatibly with the
quiver data or constructed directly from the polymatroidal tiling.

Several related directions remain open.  First, the Hilbert polynomial does
not determine the multigraded coordinate ring.  It is natural to ask for a
description of the multivariate Hilbert series, the $K$-polynomial, and the
minimal free resolution of a linked projective space, and to determine which
of these finer invariants are still independent of the linked net.  Second,
one may replace one-dimensional subrepresentations by subrepresentations of
higher constant dimension.  The resulting linked Grassmannians should be
compared with higher-rank Mustafin degenerations, but an analogue of Formula~\eqref{eq:poly-formula} proved here is not presently available.  
Finally, the proof suggests a broader combinatorial question:
which polyhedral subdivisions admit an inclusion--exclusion calculation of
their Hilbert polynomial governed solely by the relative Euler
characteristics of their cells?  Progress on these questions may clarify
which features of linked projective spaces belong to a more general interaction between quiver
Grassmannians, tropical convexity, and polymatroid geometry.

The paper is organized as follows.  In \autoref{sec:base-polytopes} we
review modular pairs, base and independence polytopes, Dilworth truncation,
and the Hilbert polynomials of images of rational maps.  In
\autoref{sec:quivers-mustafin-linked-nets} we recall linked nets and Mustafin
varieties, construct the Mustafin model associated to a polygon, and prove
the simple-normal-crossings property.  Finally, in
\autoref{sec:hilbert-polynomial} we identify the strata with
polymatroid-controlled images and prove the main theorem by
inclusion--exclusion.

\medskip

{\noindent\bf Acknowledgments.}  E.~Vital was funded by the Deutsche Forschungsgemeinschaft (DFG, German Research Foundation) – Project-ID 491392403 – TRR 358. F.~de León was supported by DICYT Project Aplicabilidad del algoritmo de Shor en Jacobianas de curvas singulares, Code 042632TSSA\_Ayudante, Vicerrectoría de Investigación, Innovación y Creación, Universidad de Santiago de Chile. We thank Omid Amini and Kris Shaw for discussions on the subject, and Yairon Cid-Ruiz for pointing references.

\smallskip

\noindent\textbf{Use of Large Language Models.} During the development and preparation of this article, the authors used ChatGPT (OpenAI; GPT-5.6 Sol) as interactive tools to explore proof strategies, and to assist with drafting and revision. The authors assume full responsibility for the contents of the article.

\section{Modular pairs and base polytopes}\label{sec:base-polytopes}
Let $H$ be a finite nonempty set, and let
$\BR^H$ be the vector space of maps $q:H\to\BR$. For $I\subseteq H$,
write
\(
    q(I)\coloneqq\sum_{v\in I}q(v),
\)
and denote the complement of $I$ in $H$ by $I^c\coloneqq H - I$. We equip $\BR^H$
with the coordinatewise partial order: for $x,y\in\BR^H$, we write
$x\leq y$ if $x(v)\leq y(v)$ for every $v\in H$. Finally, $2^H$
denotes the collection of subsets of $H$, for a function $\rho:2^H\rightarrow \mathbb{Z}$, its \textit{adjoint} is the function $\rho^*:2^H\rightarrow \mathbb{Z}$ with $\rho^*(I)\coloneqq \rho(H)-\rho(H - I)$. 

For each polytope $\bP\subseteq\mathbb{R}^H$, its \textit{relative interior}, denoted by
$\operatorname{relint}(\bP)$, is the set of points of $\bP$ that do not
belong to any proper face of $\bP$. Also, for each $\Lambda\subseteq \mathbb{R}^H$, denote $\bP(\Lambda)\coloneqq \bP\cap\Lambda$. 

We say that a function $\rho:2^H\to\BR$ is \textit{submodular} if
\[
    \rho(I_1)+\rho(I_2)
    \geq
    \rho(I_1\cup I_2)+\rho(I_1\cap I_2)\quad \text{ for each } \quad  I_1,I_2\subseteq H.
 \] 
If in addition $\rho$ is nonnegative, nondecreasing, and $\rho(\varnothing)=0$, we say that $\rho$ is a \textit{polymatroid
rank function} of \emph{rank} $\rho(H)$.

Let $\bk$ be a field. A polymatroid rank function $\rho$ on $H$ is
\textit{linear over $\bk$} if there are a finite-dimensional
$\bk$-vector space $A$ and a finite collection of vector subspaces $A_v\subseteq A$, for $v\in H$, such that
$\rho(I)=\dim_\bk\sum_{v\in I}A_v$ for every $I\subseteq H$. In this
case, the collection of $A_v$ for $v\in H$ is called a \textit{linear representation} of
$\rho$.

For a function $\rho:2^H\to\BR$, define
\[
    \bI_\rho \coloneqq
    \{
        q\in\BR_{\geq0}^H
        \mid \,
        q(I)\leq\rho(I)\text{ for every }I\subseteq H
    \}, \quad \text{ and } \quad
    \bP_\rho \coloneqq
    \{
        q\in\bI_\rho
        \mid \,
        q(H)=\rho(H)
    \}.
\]
When $\rho$ is a polymatroid rank function, $\bI_\rho$ and $\bP_\rho$
are its \textit{independence polytope} and \textit{base polytope}, respectively. 

\begin{lemma}\label{lem:continuous-polymatroid-base-extension}
Let $\rho$ be a 
submodular function. For each $x\in\bI_\rho$, there exists $b\in\bP_\rho$ such that $x\leq b$. If $\rho$ is integer-valued and $x\in\BZ_{\geq0}^H$, then $b$ may be chosen in $\BZ_{\geq0}^H$.
\end{lemma}

\begin{proof}
Let $\Lambda=\BR$. If $\rho$ and $x$ are integral, we may instead take
$\Lambda=\BZ$. In either case, consider
\[
    C_x^\Lambda
    \coloneqq
    \{y\in\bI_\rho\cap\Lambda^H\mid \, x\leq y\}.
\]
This set is nonempty and compact when $\Lambda=\BR$, while it is finite when $\Lambda=\BZ$. Choose $b\in C_x^\Lambda$ maximizing $b(H)$. It is enough to show that $b(H)=\rho(H)$.

We claim that for each $v\in H$, there must be a subset $I_v\subseteq H$ containing $v$ such that $b(I_v)=\rho(I_v)$. Indeed, if no such subset existed for some $v$, then all the finitely many slacks $\rho(I)-b(I)$, for $I$ containing $v$, would be positive. We could then increase the $v$-th coordinate of $b$ by a sufficiently small positive real number when $\Lambda=\BR$, or by $1$ when $\Lambda=\BZ$, and obtain a point of $C_x^\Lambda$ with larger modulus, a contradiction, proving the claim.

Finally, the family of tight sets $I_v$ is closed under unions by submodularity; see \cite[Thm.~44.2, p.~767]{Schrijver_B_2003}. Since $\bigcup_{v\in H}I_v=H$, it follows that $b(H)=\rho(H)$. 
\end{proof}

Let $\rho$ be a submodular function on $H$. The \textit{Dilworth
truncation} $\widehat\rho:2^H\to\BR$ is defined, for each nonempty
$I\subseteq H$, by
\[
    \widehat\rho(I)
    \coloneqq
    \min\Big\{
        \sum_{J\in\Pi}\rho(J)
        \mid\,
        \Pi\text{ is a partition of }I\text{ into nonempty subsets}
    \Big\},
\]
and by $\widehat\rho(\varnothing)=0$. By
\cite[Thm.~48.2, p.~822]{Schrijver_B_2003}, the function $\widehat\rho$ is
submodular.

The following observation follows directly from the definition; see \cite[Sec.~48.2, p.~822]{Schrijver_B_2003}.

\begin{lemma}\label{lem:Poli-dilwort}
Let $\rho:2^H\to\BR$ be a submodular function. Then
\[
    \bI_{\widehat\rho}
    =
    \left\{
        q\in\BR_{\geq0}^H
        \,\middle|\,
        q(I)\leq\rho(I)
        \text{ for every nonempty }I\subseteq H
    \right\}.
\]
\end{lemma}

\subsection{Hilbert polynomials of multiplicity-free varieties}
\label{subsection:hilbert-polynomials-multiplicity-free-varieties}
Let $\bk$ be an algebraically closed field. For each $v\in H$, let $V_v$ be a finite-dimensional $\bk$-vector space. For each $I\subseteq H$, put
\(
    \BP_I\coloneqq\prod_{v\in I}\BP(V_v),
\)
and let $\tau_I:\BP_H\to\BP_I$ be the canonical projection. For
$\Bn=(n_v)_{v\in H}\in\BZ^H$, define
\[
    \cO_{\BP_H}(\Bn)
    \coloneqq
    \bigotimes_{v\in H}
    \tau_v^*\cO_{\BP(V_v)}(n_v).
\]
For a closed subscheme $Y\subseteq\BP_H$, put
\[
    h_Y(\Bn)\coloneqq
    \dim_\bk H^0\bigl(Y,\cO_Y(\Bn)\bigr).
\]
For $n_v\gg0$ for every $v\in H$, this function agrees with a
numerical polynomial
\(
    P_Y(\Bn),
\)
which we call the \textit{multivariate Hilbert polynomial} of $Y$. 

For $q\in\BZ_{\geq0}^H$, set
\[
    C_q(\Bn)
    \coloneqq
    \prod_{v\in H}\binom{n_v+q_v}{q_v}
    \quad\text{ and }\quad
    B_q(\Bn)
    \coloneqq
    \prod_{v\in H}\binom{n_v+q_v-1}{q_v}.
\]
The polynomials $C_q(\Bn)$ form a basis of the space of numerical
polynomials. If we write 
\[
    P_Y(\Bn)
    =
    \sum_{q\in\BZ_{\geq0}^H}e_q(Y)C_q(\Bn),
\]
then $e_q(Y)\in \mathbb{Z}$, with $e_q(Y)\geq 0$ when $q(H)=\dim Y$ and $e_q(Y)=0$ when $q(H)>\dim Y$. 

Assume that $Y$ is a variety, that is, integral. Define
\(
    \varphi_Y(I)
    \coloneqq
    \dim\tau_I(Y)
\)
for nonempty $I\subseteq H$, and put $\varphi_Y(\varnothing)=0$. The
function $\varphi_Y$ is an integral polymatroid rank function;
see \cite[Prop.~5.1, p.~23]{Castillo_at_all_2021}. Put
\(
    \bI_Y\coloneqq\bI_{\varphi_Y}
\)
and 
\(
    \bP_Y\coloneqq\bP_{\varphi_Y}.
    \)
The \textit{multidegree support} of $Y$ is
\[
    M_Y
    \coloneqq
    \left\{
        q\in\BZ_{\geq0}^H
        \,\middle|\,
        q(H)=\dim Y\text{ and }e_q(Y)\neq0
    \right\}.
\]
Then \cite[Thm.~A, p.~2\textbf{}]{Castillo_at_all_2021} yields
\(
    M_Y=\bP_Y(\BZ^H).
\)
Put $D_Y\coloneqq\bI_Y(\BZ^H)$.

We say that $Y$ is \textit{multiplicity-free} if $e_q(Y)=1$ for every $q\in M_Y$. The $e_q(Y)$ for $q\in\BZ_{\geq0}^H$ with $q(H)=\dim Y$ are the coefficients in the expression of the Chow class of $Y$ in the basis of monomials on the hyperplane classes of the factors of $\BP_H$. So $Y$ is multiplicity-free if and only if its Chow class is a (multiplicity-free) sum of monomials. 

The following formula is known. In the multiview case, it appears in
\cite[Thm.~3.9]{Conca_Tsakiris_2022}, and in the general
multiplicity-free setting it follows from \cite[Cor.~1.3
and Prop.~2.8]{Eur_Larson_2026}. We include a short proof of it.  

\begin{prop}\label{prop:multiplicity-free-hilbert-polynomial}
Let $Y\subseteq\BP_H$ be a multiplicity-free variety. Then
\[
    P_Y(\Bn)
    =
    \sum_{q\in D_Y}B_q(\Bn).
\]
\end{prop}

\begin{proof}
The hockey-stick identity gives
\(
    C_q(\Bn)=\sum_{a\leq q}B_a(\Bn)
\)
for every $q\in\BZ_{\geq0}^H$. Therefore,
\begin{align*}
    P_Y(\Bn)
    =
    \sum_{q\in\BZ_{\geq0}^H}e_q(Y)
    \sum_{a\leq q}B_a(\Bn)
    =
    \sum_{a\in\BZ_{\geq0}^H}
    \bigg(\sum_{q\geq a}e_q(Y)\bigg)B_a(\Bn).
\end{align*}
By \cite[Thm.~6.9]{castillo2025kpolynomialsmultiplicityfreevarieties}, if $e_q(Y)\neq 0$ then $q\leq b$ for some $b\in M_Y$.
Thus, if $a\notin D_Y$, then
\[
    \sum_{q\geq a}e_q(Y)=0.
\]
On the other hand, if $a\in D_Y$, then there exists $m\in M_Y$ with
$m\geq a$ by \autoref{lem:continuous-polymatroid-base-extension}. So, \cite[Prop.~6.12(iii)]{castillo2025kpolynomialsmultiplicityfreevarieties} yields
\[
    \sum_{q\geq a}e_q(Y)=1.
\]
Substitution in the preceding expansion proves the result.
\end{proof}

\subsection{Hilbert polynomials of normal-crossings schemes}
\label{subsection:hilbert-polynomials-normal-crossings-schemes}
Let $Y$ be a reduced scheme of pure dimension $d$ and finite type over an algebraically closed field $\bk$. We say that $Y$ is \textit{normal-crossings} if, for each (closed) point $y\in Y$, there are an integer $1\leq r\leq d+1$ and an isomorphism of complete local rings
\[
    \widehat{\cO}_{Y,y} \simeq \bk\llbracket x_1,\dots,x_{d+1}\rrbracket/ (x_1\cdots x_r).
\]
If in addition the irreducible components of $Y$ are smooth, we say that $Y$ is \textit{simple normal-crossings}. In this case, each scheme-theoretic intersection
of components is either empty or smooth of the expected codimension. 

Denote by $E_Y$ its set of irreducible components, and by $Y_\Delta$ the scheme-theoretic intersection of the components in $\Delta$ for each nonempty $\Delta\subseteq E_Y$. For simplicity, put $Y_e\coloneqq Y_{\{e\}}$ for each $e\in E_Y$.

\begin{prop}\label{prop:normal-crossings-inclusion-exclusion}
Let $Y\subseteq\BP_H$ be a simple normal-crossings scheme.
Then
\[
    P_Y(\Bn)
    =
    \sum_{\substack{\varnothing\neq\Delta\subseteq E_Y\\
                     Y_\Delta\neq\varnothing}}
    (-1)^{|\Delta|-1}P_{Y_\Delta}(\Bn).
\]
\end{prop}

\begin{proof}
With the usual alternating restriction maps, the augmented
\v{C}ech-type complex
\[
    \begin{tikzcd}
        0 \arrow[r] & \cO_Y \arrow[r] & \displaystyle\bigoplus_{|\Delta|=1}\cO_{Y_\Delta} \arrow[r] & \displaystyle\bigoplus_{|\Delta|=2}\cO_{Y_\Delta}   \arrow[r] & \cdots \arrow[r] & \cO_{Y_{E_Y}} \arrow[r] & 0
    \end{tikzcd}
\]
is exact, where an empty intersection contributes zero; see
\cite[Def.~2.9, pp.~30--31]{Fujino_LMMP}, where it is called the
Mayer--Vietoris simplicial resolution. Although the reference works over
$\BC$, the same statement holds over an arbitrary algebraically closed field. Tensoring by
$\cO_{\BP_H}(\Bn)$ preserves exactness, and the additivity of Euler
characteristics gives the formula.
\end{proof}

\begin{prop}\label{prop:normal-crossings-euler-cancellation}
Let $Y\subseteq\BP_H$ be a simple normal-crossings scheme such that 
every nonempty $Y_\Delta$ is a multiplicity-free variety.
For each $a\in\BZ_{\geq0}^H$, let $K_a$ be the simplicial complex on
$E\coloneqq E_Y$ whose nonempty faces are the nonempty $\Delta\subseteq E$ such that $Y_\Delta\neq\varnothing$ and $a\in D_{Y_\Delta}$.
Then
\[
    P_Y(\Bn)
    =
    \sum_{a\in\BZ_{\geq0}^H}\chi(K_a)B_a(\Bn),
\]
where $\chi(K_a)$ is the Euler characteristic of $K_a$. If, in addition,
every $K_a$ having a nonempty face is connected and acyclic, then
\[
    P_Y(\Bn)
    =
    \sum_{a\in\bigcup_{e\in E}D_{Y_e}}
    B_a(\Bn).
\]
\end{prop}

\begin{proof}
We check first that $K_a$ is a simplicial complex. Let $\Delta\in K_a$ and $\varnothing\neq\Gamma\subseteq\Delta$. Since
$Y_\Delta\subseteq Y_\Gamma$, one has
that $Y_\Gamma\neq\varnothing$ and $\varphi_{Y_\Delta}(I)\leq\varphi_{Y_\Gamma}(I)$ for every
$I\subseteq H$. Thus $D_{Y_\Delta}\subseteq D_{Y_\Gamma}$, and hence
$\Gamma\in K_a$.

By \autoref{prop:normal-crossings-inclusion-exclusion} and
\autoref{prop:multiplicity-free-hilbert-polynomial}, reordering the finite
sums gives
\begin{align*}
    P_Y(\Bn) & =
    \sum_{\substack{\varnothing\neq\Delta\subseteq E\\
                     Y_\Delta\neq\varnothing}}
    (-1)^{|\Delta|-1}\sum_{a\in D_{Y_\Delta}}B_a(\Bn)
    =
    \sum_{a\in\BZ_{\geq0}^H}\Big(
        \sum_{\substack{\Delta\in K_a\\\Delta\neq\varnothing}}
        (-1)^{|\Delta|-1}
    \Big)B_a(\Bn)
    =
    \sum_{a\in\BZ_{\geq0}^H}\chi(K_a)B_a(\Bn).
\end{align*}

Now assume that every $K_a$ having a nonempty face is connected and acyclic then $\chi(K_a)=1$, see \cite[Thm.~2.44, p.~146]{Hatcher_2002}.
If it has only the empty face, then $\chi(K_a)=0$. Since $K_a$ is simplicial, it has a nonempty face if and only if it contains a singleton, which is equivalent to $a\in\bigcup_{e\in E}D_{Y_e}$. Thus, the last formula follows.
\end{proof}

\begin{example}\label{Exa: K_a_disconnected} We give an example, where the first formula in \autoref{prop:normal-crossings-euler-cancellation} applies, but does not reduce to the second.
Let $H\coloneqq [2] = \{1,2\}$. Put $\BP_H\coloneqq\BP^1\times \BP^1$ with coordinates $([x_0:x_1],[y_0:y_1])$. Let 
\[
Y_1\coloneqq V(x_0),\quad Y_2\coloneqq V(y_0),\quad Y_3\coloneqq V(x_0y_1+x_1y_0+x_1y_1),
\quad\text{and} \quad Y\coloneqq Y_1\cup Y_2\cup Y_3.
\]
Then $Y$ is simple normal-crossings, with
\[
    Y_1\cap Y_2  =([0:1],[0:1]), \quad Y_1\cap Y_3  =([0:1],[1:-1]), \quad Y_2\cap Y_3  =([1:-1],[0:1]),
\]
and $Y_1\cap Y_2\cap Y_3=\varnothing$. 
Furthermore,
\[
    D_{Y_1}=\{(0,0),(0,1)\}, \quad  D_{Y_2}=\{(0,0),(1,0)\}, \quad \text{and} \quad D_{Y_3}=\{(0,0),(1,0),(0,1)\},
\]
and the corresponding multivariate Hilbert polynomials are
\[
P_{Y_1}(n_1,n_2)=1+n_2,
\quad
P_{Y_2}(n_1,n_2)=1+n_1,
\quad \text{and} \quad 
P_{Y_3}(n_1,n_2)=1+n_1+n_2.
\]
Clearly, the nonempty strata $Y_\Delta$ are multiplicity-free. 

Since $0\in D_{\Delta}$ for all $\Delta\subseteq [3]$ with $Y_\Delta\neq\varnothing$, the simplicial complex $K_{(0,0)}$ has $3$ vertices, $3$ edges but no $2$-faces. It is thus cyclic, homotopic to $S^1$, satisfying 
\(
\chi(K_{(0,0)})
=0.
\)
Also, 
\(
K_{(1,0)}=\big\{\{2\},\{3\}\big\}
\)
and
\(
K_{(0,1)}=\big\{\{1\},\{3\}\big\},
\)
both consisting of two disconnected vertices. Then $\chi(K_{(1,0)})=\chi(K_{(1,0)})=2$.

By \autoref{prop:normal-crossings-euler-cancellation},
the Hilbert polynomial of $Y$ is 
\[
    P_Y(n_1,n_2) = 0\cdot B_{(0,0)} +2B_{(1,0)} +2B_{(0,1)} = 2n_1+2n_2.
\]
On the other hand, \[
\sum_{a\in\cup D_{Y_i}}B_a(n)
=
B_{(0,0)}+B_{(1,0)}+B_{(0,1)}=1+n_1+n_2.
\]
\end{example}

\subsection{Polytopes associated to vector spaces}
\label{subsection:polytopes-vector-spaces}
Let $\bk$ be an algebraically closed field. For each $v\in H$, let $V_v$ be a finite-dimensional $\bk$-vector space. 
For each $I\subseteq H$, put $V_I\coloneqq\bigoplus_{v\in I}V_v$, and let
$\theta_I\colon V_H\to V_I$ be the corresponding
projection map. By convention, $V_\varnothing=(0)$.

Let $W\subseteq V_H$ be a vector subspace. Define
\(
    W_I\coloneqq\theta_I(W),
\)
and
\(
    \rho_W(I)\coloneqq\dim_\bk W_I.
\)
Then, $\rho_W$ is an integral polymatroid rank function. Write
$\bI_W\coloneqq\bI_{\rho_W}$ and $\bP_W\coloneqq\bP_{\rho_W}$.
The diagonal action of $(\bk^*)^H$ on $V_H$
preserves the rank function: $\rho_{\lambda W}=\rho_W$ for every
$\lambda\in(\bk^*)^H$.

Assume from now on that every coordinate projection
$p_v\coloneqq\theta_v|_W\colon W\to V_v$ is nonzero. For each nonempty $I\subseteq H$, let
$Y_{W,I}$ be the scheme-theoretic image of the rational map
\[
    \begin{tikzcd}
        \BP(W)  \arrow[r, dashed] & \displaystyle\prod_{v\in I}\BP(V_v)
    \end{tikzcd}
\]
induced by the coordinate projections, and put $Y_W\coloneqq Y_{W,H}$.
By \cite[Thm.~1.1, p.~4191]{Li_2018}, the variety $Y_W$ is multiplicity-free. Put 
\[
    \varphi_W\coloneqq\varphi_{Y_W},\quad M_W\coloneqq M_{Y_W},\quad \text{ and } \quad D_W\coloneqq D_{Y_W}.
\]

\begin{prop}\label{prop:projection-rank-dilworth-truncation}
For each $v\in H$, let
\(
    A_v
\)
be the image of the dual map 
\(
p_v^*:V_v^*\to W^*.
\)
Then the $A_v$, for $v\in H$, form a linear representation of $\rho_W$.
Furthermore, for general $[w]\in\BP(W)$ 
the intersections $A_v\cap w^\perp$, for $v\in H$, form a linear representation of $\varphi_W$. Also,
\[
    \varphi_W=\widehat{\rho_W-1},
\]
where $(\rho_W-1)(I)\coloneqq\rho_W(I)-1$.
\end{prop}

\begin{proof}
For each $I\subseteq H$, the sum of the dual maps $p_v^*$ for $v\in I$ has image $\theta_I(W)^*$. Hence $\rho_W(I)=\dim_\bk\sum_{v\in I}A_v$, proving the first statement.

As for the second statement, choose $w\in W$ general with $p_v(w)\neq0$ for every $v\in H$.
For each $I\subseteq H$, let
\[
    F_I(w)\coloneqq\bigcap_{v\in I}p_v^{-1}\big(\langle p_v(w)\rangle\big)\subseteq W.
\]
The fiber through $[w]$ of the rational map defining $Y_{W,I}$ is a
dense open subset of $\BP\big(F_I(w)\big)$. Furthermore, the annihilator of
$F_I(w)$ is 
\[
    \sum_{v\in I}p_v^*\big(p_v(w)^\perp\big) =\sum_{v\in I}(A_v\cap w^\perp).
\] Therefore,
\[
    \varphi_W(I)
    =\dim W-\dim F_I(w)
    =\dim\sum_{v\in I}(A_v\cap w^\perp),
\]
proving the second statement.

Since $w^\perp\subset W^*$ is
a general hyperplane, it follows from 
\cite[Thm.~2.5, p.~397]{Raz2019} that the set $\lbrace A_v\cap w^\perp\,\vert\, v\in H\rbrace$ is a linear representation of the Dilworth truncation of $\rho_W-1$. Thus, by the second statement, we have \( \varphi_W=\widehat{\rho_W-1}\), as required. 
\end{proof}

The following proposition specializes \autoref{prop:multiplicity-free-hilbert-polynomial}, and recovers
Conca--Tsakiris~\cite[Thm.~3.9]{Conca_Tsakiris_2022}.

\begin{prop}\label{prop:general-vector-space-hilbert-polynomial}
Let $W\subseteq V_H$ be a vector subspace with nonzero coordinate projections. Then
\[
    D_W
    =
    \left\{
        a\in\BZ_{\geq0}^H
        \,\middle|\,
        a(I)\leq\rho_W(I)-1
        \text{ for every nonempty }I\subseteq H
    \right\},
\]
and
\[
    P_{Y_W}(\Bn)
    =
    \sum_{a\in D_W}B_a(\Bn).
\]
\end{prop}

\begin{proof} The first formula follows from \autoref{prop:projection-rank-dilworth-truncation} and \autoref{lem:Poli-dilwort}, whereas the second is \autoref{prop:multiplicity-free-hilbert-polynomial}, using that $Y_W$ is multiplicity-free.
\end{proof}

The following proposition characterizes geometrically $D_W$.

\begin{prop}\label{prop:interior-criterion} Let $W\subseteq V_H$ be a vector subspace with nonzero coordinate projections. Let $a\in\BZ^H_{\geq0}$. Then, for every positive $\varepsilon<1/|H|$,
\[
    a\in D_W
    \quad\Longleftrightarrow\quad
    \bP_W\cap C_{a,\varepsilon}\neq\varnothing,    
\]
where 
\[
C_{a,\varepsilon}
\coloneqq a+\varepsilon\mathbf 1 + \{q\in\BR_{\geq 0}^H\mid  q(H)=\dim W-a(H)-\varepsilon|H|\}.
\]
\end{prop}

\begin{proof} 
If $a\in D_W$, then
\autoref{prop:general-vector-space-hilbert-polynomial} gives
$a(I)\leq\rho_W(I)-1$ for each nonempty $I\subseteq H$.
Thus $a(I)+\varepsilon|I|<\rho_W(I)$ for every such $I$, and hence $a+\varepsilon\mathbf1\in\bI_W$. Now, 
\autoref{lem:continuous-polymatroid-base-extension} gives
$b\in\bP_W$ with $b\geq a+\varepsilon\mathbf1$. Since
$b(H)=\rho_W(H)=\dim W$, this inequality says exactly that
$b\in C_{a,\varepsilon}$. Conversely, if
$b\in\bP_W\cap C_{a,\varepsilon}$, then
$a(I)<b(I)\leq\rho_W(I)$ for each nonempty $I\subseteq H$. Since $a(I)$ and $\rho_W(I)$ are integral,
$a(I)\leq\rho_W(I)-1$, and hence $a\in D_W$, finishing the proof.
\end{proof}

\section{Quiver representations, Mustafin varieties, and linked nets}
\label{sec:quivers-mustafin-linked-nets}

\subsection{Quiver representations}\label{subsection:quiver-representations}
A \textit{quiver} is a directed graph. All the quivers we consider have neither loops nor multiple arrows, so we may as well assume this. For a quiver $Q$, we denote its set of vertices by $Q_0$ and its set of arrows by $Q_1$. As there are no multiple arrows, we may view $Q_1$ in $Q_0\times Q_0$, associating to an arrow with tail $u$ and head $v$ the pair $(u,v)$, and denoting it by $uv$ for short. Likewise, a path in $Q$ starting at $u_0$, passing through $u_1,\dots,u_{m-1}$ in this order with no intermediate vertices, and ending at $u_m$, will be denoted $u_0u_1\cdots u_{m-1}u_m$ and said to have \emph{length} $m$. 

Let $\bk$ be a field. A (\textit{quiver}) \textit{representation} $\mathfrak{V}$ of a quiver $Q$ is the data of a $\bk$-vector space $V_v$ for each $v\in Q_0$, and a $\bk$-linear map $\varphi^u_v:V_u\to V_v$ for each $uv\in Q_1$. Given it, for each path $\gamma=u_0\cdots u_m$ in $Q$, we denote by $\varphi_\gamma$ the composition $\varphi^{u_{m-1}}_{u_m}\cdots \varphi^{u_0}_{u_1}$. If $\gamma$ is the trivial path, that is, $m=0$, we put $\varphi_\gamma=\mathrm{id}_{V_{u_{0}}}$ by convention. 

Let $H$ be a set of vertices of $Q$. We say $\mathfrak{V}$ is \textit{generated} by $H$ if, for each vertex $v$ of $Q$, there exist $w\in H$ and a path $\gamma$ connecting $w$ to $v$ such that $\varphi_\gamma$ is surjective.
If $H$ is finite, we say $\mathfrak{V}$ is \textit{finitely generated}.

\subsection{Mustafin varieties}\label{subsection:mustafin}

Let $R$ be a discrete valuation ring, $\BK$ its field of fractions, $\pi$ a uniformizer parameter, and $\bk$ the residue field.  Let $r$ be a positive integer. A \textit{lattice} in $\BK^r$ is a (necessarily free) $R$-submodule of rank $r$. For each lattice $\Lambda$, denote by $\overline{\Lambda}$ the quotient $\Lambda/\pi\Lambda $.
 
Two lattices $\Lambda_1,\Lambda_2\subseteq \BK^r$ are \textit{equivalent} if $\Lambda_1 = c\Lambda_2$ for some $c\in \BK^*$. Let $\mathcal{B}^0_r$ be the set of equivalent classes of lattices in $\BK^r$, the Bruhat--Tits building of $\mathrm{PGL}(r)$; see \cite[Def.~2.22, p.~13]{Gortz_2010}. Two lattice classes are called \textit{adjacent} if there exist  representatives $\Lambda_1$ for one and $\Lambda_2$ for the other such that $\pi \Lambda_2 \subseteq \Lambda_1 \subseteq \Lambda_2$. By convention, a lattice class is not adjacent to itself. Adjacency is a symmetric, but neither reflexive nor transitive relation.

A subset $\Gamma\subseteq \mathcal{B}^0_r$ is a \emph{lattice configuration}. It is \textit{convex} if $[\pi^a\Lambda_1\cap \pi^b\Lambda_2]\in \Gamma$ for each $[\Lambda_1],[\Lambda_2]\in\Gamma$ and $a,b\in \mathbb{Z}$. The \textit{convex hull} of $\Gamma$ is the smallest convex subset of $\mathcal{B}^0_r$ containing $\Gamma$, denoted by $\mathrm{conv}(\Gamma)$. If $\Gamma$ is finite, so is $\mathrm{conv}(\Gamma)$. \par

Let $\Gamma\subseteq \mathcal{B}^0_r$ be a finite lattice configuration. Let $\Lambda_v$ for $v\in H$ be a collection of pairwise nonequivalent lattices indexed by a set $H$ representing all the classes in $\Gamma$. For short, we write $\Gamma=\{[\Lambda_v]\,|\,v\in H\}$. The $\BP(\Lambda_v)$ for $v\in H$ are projective spaces over $R$ whose generic fibers are canonically isomorphic to $\BP^{r-1}_{\BK}$ via the inclusion $\Lambda_v\hookrightarrow\BK^r$. The open embeddings $\BP^{r-1}_{\BK}\hookrightarrow \BP(\Lambda_v)$ give rise to a map
\[
    \begin{tikzcd}
         \BP^{r-1}_{\BK} \arrow[r] & \displaystyle \prod_{v\in H}\BP(\Lambda_v),
    \end{tikzcd}
\]
where the product is fibered over $R$.

\begin{definition}\label{def:mustafin-scheme}\cite[Def 1.1, p.~758]{Cartwright2011}
    The schematic closure of the image of the above map is called the \textit{Mustafin scheme}
    associated to $\Gamma$, and denoted $\cM(\Gamma)$. Its special fiber, denoted $\cM(\Gamma)_\bk$, is the associated \emph{Mustafin variety}. 
\end{definition}

Notice that the generic fiber of $\cM(\Gamma)$ is naturally isomorphic to  $\BP^{r-1}_{\BK}$ embedded diagonally in the product of copies of $\BP^{r-1}_{\BK}$ indexed by $H$. Also, the Mustafin scheme does not depend on the choice of representative lattices $\Lambda_v$. 

In \cite{Cartwright2011}, $\cM(\Gamma)$ itself is called a Mustafin variety. It is an unnecessary misnomer that we do not adopt. The scheme $\cM(\Gamma)$ has also been called a \emph{Deligne scheme} if $\Gamma$ is convex, which is the case of interest to us. By \cite[Thm.\ 2.3, p.\ 763]{Cartwright2011}, the scheme $\cM(\Gamma)$ is integral, normal and Cohen--Macaulay, and $\cM(\Gamma)_\bk$ is reduced, Cohen--Macaulay and connected. 

Let $\Gamma=\{[\Lambda_v]\,|\,v\in H\}$ be a finite convex lattice configuration in $\mathcal{B}^0_r$. The Deligne scheme $\cM(\Gamma)$ associated to $\Gamma$ is isomorphic to a certain quiver Grassmannian over $R$, that we now describe. 

First, we define the quiver $Q(\Gamma)$ whose set of vertices $Q(\Gamma)_0$ is $H$ and whose set of arrows is the set of pairs of adjacent lattices, that is,
\[
    Q(\Gamma)_1\coloneqq \left\lbrace (u,v)\in H^2  \,\big\vert\, [\Lambda_u]\text{ and }[\Lambda_v] \text{ are adjacent}\right\rbrace.
\] 
Clearly, $Q(\Gamma)$ is a finite quiver with neither loops nor multiple edges.

Second, we define a quiver representation $\mathfrak{V}(\Gamma)$ of $Q(\Gamma)$. For each pair of vertices $v,u\in H$, let $p_{v,u}$ be the minimum integer such that $\pi^{p_{v,u}}\Lambda_v\subseteq \Lambda_u$. Denote by $G^v_{u}:\Lambda_v\to  \Lambda_u$ the $R$-linear map given by multiplying by $\pi^{p_{v,u}}$, and by $g^v_u: \overline{\Lambda_v}\to  \overline{\Lambda_u}$ the induced map. The data consisting of the $\Lambda_v$ (resp.~$\overline{\Lambda_v}$) for $v\in H$ and the maps $G^v_u$ (resp.~$g^v_u$) for $vu\in Q(\Gamma)_1$ is a quiver representation of $Q(\Gamma)$ in the category of $R$-modules (resp.~$\bk$-vector spaces), denoted by $\mathfrak{V}(\Gamma)$ (resp.~$\mathfrak{V}(\Gamma)_\bk$). 

Finally, let $\BLP\big(\mathfrak{V}(\Gamma)\big)$ (resp.~$\BLP\big(\mathfrak{V}(\Gamma)_\bk\big)$) be the quiver Grassmannian of subrepresentations of $\mathfrak{V}(\Gamma)$ (resp.~$\mathfrak{V}(\Gamma)_\bk$) by saturated $R$-submodules of rank $1$ (resp.~$\bk$-subspaces of dimension~$1$). Though $\mathfrak{V}(\Gamma)$ depends on the choice of representative lattices $\Lambda_v$, the $R$-scheme $\BLP\big(\mathfrak{V}(\Gamma)\big)$ does not. Clearly, $\BLP\big(\mathfrak{V}(\Gamma)_\bk\big)$ is the special fiber $\BLP\big(\mathfrak{V}(\Gamma)\big)_\bk$ of $\BLP\big(\mathfrak{V}(\Gamma)\big)$. Furthermore, $\BLP\big(\mathfrak{V}(\Gamma)\big)=\cM(\Gamma)$ as $R$-schemes by \cite[Cor.~4.6, p.~194]{HahnLi2021}.  

\subsection{Linked nets}\label{subsec:quivers-linked-nets}

Let $Q$ be a quiver. Let $T$ be a finite partition of the arrow set of $Q$. Each part $\mathfrak{a}\in T$ is called an \textit{arrow type}, and we say $a\in \mathfrak{a}$ \textit{has} or \textit{is of} type $\mathfrak{a}$.

Given a path $\gamma$ in $Q$ and an arrow type $\mathfrak{a}$, we denote by $\mathbf{t}_{\gamma}(\mathfrak{a})$ the number of arrows of that type that the path contains. We call the corresponding function $\mathbf{t}_{\gamma}\colon T\to\BZ$ the \textit{type} of $\gamma$, and the collection of arrow types $\lbrace \mathfrak{a}\,\vert\,\mathbf{t}_{\gamma}(\mathfrak{a})>0\rbrace$ its \textit{essential type}. The path $\gamma$ is called $\textit{admissible}$ if $\mathbf{t}_{\gamma}(\mathfrak{a})=0$ for some $\mathfrak{a}\in T$, and \textit{simple} if $\mathbf{t}_{\gamma}(\mathfrak{a})\leq 1$ for every $\mathfrak{a}\in T$. 

Let $n\coloneqq |T|-1$. We say that $T$ is a $\BZ^{n}$-\textit{structure} on $Q$ if the following four conditions are satisfied: 
\begin{enumerate}[label=(Q\arabic*), ref=(Q\arabic*), series=Q]
    \item\label{Q1} There is exactly one arrow of each type leaving each vertex.
    \item\label{Q2} Each vertex is connected to each other by an admissible path.
    \item\label{Q3} Each maximal simple path is a circuit. 
    \item\label{Q4} Two admissible paths leaving the same vertex arrive at the same vertex if and only if they are of the same type.
\end{enumerate}
A quiver with a $\BZ^n$-structure is called a $\BZ^n$-\textit{quiver}.

A maximal simple path is nonadmissible, and is called a \textit{minimal circuit}. Two distinct vertices connected by a simple path are called \textit{neighbors}. If $v_1$ and $v_2$ are neighbors, and $I$ is the essential type of a simple path connecting $v_1$ to $v_2$, we write $v_2=I\cdot v_1$.

The \textit{hull} of a collection of vertices $H$ of a $\BZ^n$-quiver $Q$ is the set $P(H)$ of all vertices $v$ of $Q$ such that for each arrow type $\mathfrak{a}$ there are $z\in H$ and a path $\gamma$ connecting $z$ to $v$ not containing any arrow of type $\mathfrak{a}$; see \cite[Def.~5.3, p.~538]{Esteves22072025}. If $P(H)=H$ then, for each vertex $u$ of $Q$, there is a unique vertex $w_u\in H$, called the \textit{shadow} of $u$ in $H$, such that for each $v\in H$ there is an admissible path connecting $v$ to $u$ passing through $w_u$; see \cite[Prop.~5.7, p.~540]{Esteves22072025}. 

Let $\bk$ be a field, $Q$ a $\BZ^n$-quiver and $\mathfrak{V}$ a representation of $Q$ by vector spaces over $\bk$. We say that $\mathfrak{V}$ is a \textit{linked net} over $Q$ if it satisfies the following conditions:
\begin{enumerate}[label=(N\arabic*), ref=(N\arabic*), series=N]
    \item\label{N1} If $\gamma_1$ and $\gamma_2$ are two paths connecting the same two vertices and $\gamma_2$ is admissible then $\varphi_{\gamma_1}$ is a scalar multiple of $\varphi_{\gamma_2}$.
    \item\label{N2} $\varphi_{\gamma}=0$ for each minimal circuit $\gamma$.
    \item\label{N3} If $\gamma_1$ and $\gamma_2$ are two admissible paths leaving the same vertex with no arrow type in common, then $\ker(\varphi_{\gamma_1})\cap \ker(\varphi_{\gamma_2})=0$. 
\end{enumerate}

Assume $\mathfrak{V}$ is a linked net. We say that $\mathfrak{V}$ is \textit{exact} if for each two vertices $u_1$ and $u_2$ of $Q$ which are neighbors, and admissible paths $\gamma_1$ and $\gamma_2$ connecting $u_1$ to $u_2$ and $u_2$ to $u_1$, respectively, we have $\im(\varphi_{\gamma_1})=\ker(\varphi_{\gamma_2})$.

If $\mathfrak{V}$ is a finitely generated representation by vector spaces of the same finite dimension, we denote by $\BLP(\mathfrak{V})$ the quiver Grassmannian of its subrepresentations by subspaces of dimension~$1$. When $\mathfrak{V}$ is a linked net, this projective variety is called the \emph{linked projective space} associated to $\mathfrak{V}$; see \cite[Def.~5.3, p.~18]{Esteves2025}.

\subsection{Linked projective spaces are locally Mustafin varieties}
\label{subsec:locally-mustafin}

Let $Q$ be a $\BZ^n$-quiver, and $H$ a set of vertices of $Q$ such that $P(H)=H$. For each $v\in H$, denote by  $R_v^H$ the set of vertices of $Q$ whose shadow in $H$ is $v$; we write $R_v$ when there is no ambiguity. We call $R_v^H$ the \textit{shadow region} of $v$ in $H$. These regions form a partition of set of vertices $Q_0$.
  
A \textit{polygon} in $Q$ is a nonempty collection of pairwise neighboring vertices of $Q$. A polygon with $m$ vertices is called a $m$-$\textit{gon}$. Given a vertex $v$ of $Q$, a way of constructing a $m$-gon containing $v$ is to pick an ordered $m$-partition $T=I_1\cup\cdots\cup I_m$ of the set of arrow types $T$ of $Q$, and put $v_1\coloneqq v$ and $v_{i+1}\coloneqq I_i\cdot v_{i}$ for $i=1,\dots,m-1$; then $\lbrace v_1,\dots,v_m\rbrace$ is a $m$-gon. We say that $v_1,\dots,v_m$ form an \textit{oriented $m$-gon}. All $m$-gons containing $v$ are of this form, for all choices of $m$-partitions of $T$.

\begin{prop}\label{prop:polygon-linked-projective-space-mustafin}
Let $\mathfrak{V}$ be an exact finitely generated linked net of vector spaces of finite dimension $r>0$ over a $\BZ^n$-quiver $Q$. If $\mathfrak{V}$ is generated by a polygon, then 
there is a finite convex lattice configuration $\Gamma\subseteq\mathcal{B}^0_r$ such that
\[
\BLP(\mathfrak{V})=\cM(\Gamma)_\bk.
\]
\end{prop}

\begin{proof}
Let $v_1,\dots,v_m$ be vertices of $Q$ forming an oriented $m$-gon $\Delta$ generating $\mathfrak V$. 
For convenience, put $v_{m+1}\coloneqq v_1$. Since $\mathfrak{V}$ is generated by $\Delta$, it follows from \cite[Prop.~10.1, p.~550]{Esteves22072025} that $\mathfrak{V}$ is the direct sum of exact linked nets $\mathfrak{V}_1,\dots,\mathfrak{V}_{r}$. Since $\mathfrak{V}$ is generated by $\Delta$, so are the $\mathfrak{V}_j$. By \cite[Thm.~7.8, p.~544]{Esteves22072025}, for each $j\in[r]$, the net $\mathfrak{V}_j$ is generated by $\{v_{\ell_j}\}$ for a unique $\ell_j\in[m]$; let $s_j\in V_{v_{\ell_j}}$ be a generator. Reorganize so that $\ell_1\leq\ell_2\leq\cdots\leq\ell_r$. For each $\ell\in[m]$, let $r_\ell\coloneqq |\lbrace j~\vert~ \ell_j=\ell\rbrace|$. We have $\sum r_\ell=r$. The $s_j$ induce a basis for $V_{v_i}$ for each $i\in[m]$, and thus a decomposition $V_{v_i}=V_{i,1}\oplus\cdots\oplus V_{i,m}$, where $V_{i,\ell}$ is the subspace freely generated by the $\varphi^{v_\ell}_{v_i}(s_j)$ for $j$ with $\ell_j=\ell$, for $\ell\in[m]$. Notice that $V_{i,\ell}$ has dimension $r_\ell$. 

For each $i\in[m]$, the map $\varphi^{v_i}_{v_{i+1}}$ can thus be represented by a diagonal matrix $M_i$ of size $r$. If $i<m$, all of its diagonal entries are $1$ but for those in position $r_1+\cdots+r_i+j$ for $j=1,\dots,r_{i+1}$, which are $0$. The matrix $M_m$ has all of its diagonal entries $1$ but for those in position $1,\dots,r_1$, which are $0$. 

Put $R\coloneqq\bk\llbracket T \rrbracket$. Let $\overline{x}\in\bk$ denote the residue of $x\in R$. 
Define, for each $i\in[m]$, 
$$
\alpha_i\coloneqq (\alpha_{i,1},\dots,\alpha_{i,r}),\quad \text{ where }
\alpha_{i,j}=\begin{cases}
    1 &\text{if }j\leq r_1+\cdots+r_i,\\
    0 &\text{otherwise,}
\end{cases}, \quad \text{ and }
\Lambda_i\coloneqq\bigoplus_{j=1}^{r} T^{-\alpha_{ij}}R.
$$
Put $\Gamma\coloneqq\lbrace [\Lambda_1],\dots,[\Lambda_m]\rbrace$. Let $\Lambda\coloneqq R^{r}$.

First, we verify that $\Gamma$ is convex. All lattices in $\Gamma$ are in the same apartment. Hence, it follows from \cite[Lem.~4.1, p.~773]{Cartwright2011} that $\Gamma$ is convex if and only if the classes $[\alpha_1],\dots,[\alpha_m]$ in the tropical projective space $\mathbb{TP}^{r-1}$ are the lattice points of their tropical convex hull 
\[
    C\coloneqq \mathrm{tconv}([\alpha_1],\dots,[\alpha_m])\subseteq \mathbb{TP}^{r-1}.
\]
Let $[\beta]\in C$ be a lattice point, $\beta=(\beta_1,\dots,\beta_r)\in\BZ^{r}$. Since $[\beta]$ is in $C$, the representative $\beta$ can be written as a tropical linear combination of the $\alpha_i$, that is, there are $\lambda_1,\dots,\lambda_m\in\BR$ such that
$$
\beta=\lambda_1\odot\alpha_1\oplus\cdots\oplus\lambda_m\odot\alpha_m,
$$
where $\odot$ is the tropical product and $\oplus$ is the tropical sum. 
It follows that for each $i\in[m]$ and $j=r_1+\cdots+r_{i-1}+1,\dots,r_1+\cdots+r_i$,
$$
\beta_j=\tau_i\coloneqq \min\lbrace \lambda_1,\dots,\lambda_{i-1}, \lambda_i+1,\dots,\lambda_m+1 \rbrace.
$$
Notice that
$$
\tau_1\geq\tau_2\geq\cdots\geq\tau_m\geq\tau_1-1.
$$
Let $\ell\in\{1,\dots,m\}$ be the largest integer such that $\tau_\ell=\tau_1$. Then $\tau_i=\tau_1$ for $i\leq \ell$ and $\tau_i=\tau_1-1$ for $i>\ell$. It follows that 
$\beta=(\tau_1-1)\odot\alpha_\ell$. Thus $[\beta]=[\alpha_\ell]$, as required.

It remains to identify the associated Mustafin variety. Let $Q(\Gamma)$ be the quiver associated to $\Gamma$. Any two lattices in $\Gamma$ are adjacent. Indeed, if $i\geq\ell$ then $T\Lambda_\ell\subseteq\Lambda_i\subseteq\Lambda_\ell$. For each $i\in[m]$, let $g_i$ be the diagonal matrix whose entries in the diagonal are $1$ but those in position $j$ for $j\in [r_1+\cdots +r_{i}]$, which are $T$. Notice that $g_i\Lambda_i=\Lambda$. For $i,\ell\in[m]$, let $p_{i,\ell}$ be the minimum integer such that $T^{p_{i,\ell}}\Lambda_i\subseteq\Lambda_\ell$, and put $g_{i,\ell}\coloneqq g_{\ell}T^{p_{i,\ell}}g_{i}^{-1}$. We have  
\[
    p_{i,\ell}=\left\lbrace
    \begin{array}{cc}
        0 & i\leq \ell, \\
        1 & i>\ell.
    \end{array}\right.
\] 
By \cite[Cor.~4.6, p.~194]{HahnLi2021}, 
\[
\cM(\Gamma)_\bk=\bigg\lbrace ([s_1],\dots,[s_m])\in \prod_{i=1}^m \BP^{r-1}_{\bk}\,\,\Big|\,\, \overline{g_{i,\ell}}(s_i)\wedge s_\ell=0 \text{ for all } i,\ell\in[m]\bigg\rbrace.
\]
Thus, we need only show that $\overline{g_{i,i+1}}=M_i$ for each $i\in[m]$, where $g_{m,m+1}\coloneqq g_{m,1}$. Indeed, for $i\in[m-1]$, we have 
\[
    \overline{g_{i,i+1}}=\overline{g_{i+1}}\,\overline{g_i}^{-1}=M_i,
\]
whereas for $i=m$ we have
\[
    \overline{g_{m,1}}=\overline{g_{1}Tg_m^{-1}}=M_m.
\]
This proves that $\cM(\Gamma)_\bk=\BLP(\mathfrak{V})$ and completes the proof.
\end{proof}

\begin{thm}\label{thm:linked-projective-space-snc}
    Let $\mathfrak{V}$ be a nontrivial exact finitely generated linked net of $\bk$-vector spaces over a $\BZ^n$-quiver. Then every point on $\BLP(\mathfrak{V})$ admits a Zariski open neighborhood isomorphic to an open subvariety of a Mustafin variety associated to a convex lattice configuration. In particular, $\BLP(\mathfrak{V})$ is simple normal-crossings.
\end{thm}

\begin{proof}
  Let $\mathfrak{W}\in \BLP(\mathfrak{V})$. By \cite[Thm.~3.6, p.~9]{Esteves2025}, there is a polygon $\Delta$ in $Q$ generating $\mathfrak{W}$. Since $P(\Delta)=\Delta$, to each $v\in Q_0$ we may associate its shadow $w_v$ in $\Delta$. Then we may consider the representation $\mathfrak{V}_{\Delta}$ associating to each $v\in Q_0$ the vector space $V_{w_v}$, and to each arrow $uv\in Q_1$ the $\bk$-linear map $\psi^u_v\colon V_{w_u}\rightarrow V_{w_v}$, which is zero if there is no admissible path from $w_u$ to $v$ through $u$, and equal to $\varphi_{\gamma}$ otherwise, where $\gamma$ is any chosen admissible path from $w_u$ to $w_v$, as in \cite[Sec.~7, pp.~22--25]{Esteves2025}. 
  
  It follows from \cite[Prop.~7.8, p.~25]{Esteves2025} that $\mathfrak{V}_{\Delta}$ is an exact linked net of vector spaces over $Q$ generated by $\Delta$, and hence $\BLP(\mathfrak{V}_{\Delta})$ is a Mustafin variety associated to a convex lattice configuration by \autoref{prop:polygon-linked-projective-space-mustafin}. Thus $\BLP(\mathfrak{V}_{\Delta})$ is simple normal-crossings by \cite[Sec.~5, p.~167]{Faltings2001}. It is now enough to observe that it follows from the proof of \cite[Thm.~8.2, p.~28]{Esteves2025} that there is a Zariski open neighborhood of $\mathfrak{W}$ in $\BLP(\mathfrak{V})$ which is isomorphic to an open subscheme of $\BLP(\mathfrak{V}_{\Delta})$.
\end{proof}

\begin{rem} That $\BLP(\mathfrak{V})$ is simple normal-crossings may also be derived from the local equations found in the proof of \cite[Thm.~8.2, p.~28]{Esteves2025}.
\end{rem}

\section{Hilbert polynomial of a linked projective space.}\label{sec:hilbert-polynomial}

\subsection{Linked nets and polytopes}\label{subsec:linked-nets-polytopes}
Let $\bk$ be an algebraically closed field. Let $\mathfrak{V}$ be a nontrivial finitely generated exact linked net of $\bk$-vector spaces of dimension $r+1$ over a $\BZ^n$-quiver $Q$, and let $H$ be its minimum set of generators.

Let $V_v$ be the vector space associated to $v\in H$ by $\mathfrak{V}$, and put $V_H\coloneqq\bigoplus_{u\in H}V_u$. For each path $v,u\in H$, let $\varphi^v_u$ be the map associated by $\mathfrak{V}$ to a chosen path $\gamma$ in $Q$ conecting $v$ to $u$. For each polygon $\Delta\subseteq H$ and $v\in\Delta$, set
\[
V^\Delta_v\coloneqq\bigcap_{u\in H - R_v^\Delta}\ker(\varphi_u^v)\subseteq V_v,
\]
and put $V^\Delta\coloneqq\bigoplus_{v\in\Delta}V^\Delta_v$. Define $\Psi_H^\Delta:V^\Delta\to V_H$ by
\[
    \Psi_H^\Delta(s_v\,|\,v\in\Delta)=\big(\varphi_u^{w_u}(s_{w_u})\,|\,u\in H\big),
\]
where $w_u$ is the shadow of $u$ in $\Delta$. Put $W^\Delta\coloneqq\im(\Psi_H^\Delta)\subseteq V_H$, and let $\rho_\Delta\coloneqq\rho_{W^\Delta}$ and $\bP_\Delta\coloneqq\bP_{\rho_\Delta}$. Since the diagonal action of $(\bk^*)^H$ on $V_H$ preserves the rank function, $\rho_\Delta$ does not depend on the choice of the $\varphi^u_v$. 
 
By \cite[Lem.\ 4.4(3)]{Chow.class}, 
every coordinate projection $p_v\colon W^\Delta\to V_v$ is nonzero. Let $X^\Delta$ be the scheme-theoretic image of the induced rational map, 
\[
    \begin{tikzcd}
        \Phi_{\Delta}:\mathbb{P}(W^\Delta) \arrow[r, dashed] & \displaystyle\BP_H\coloneqq\prod_{u\in H}\mathbb{P}(V_u).
    \end{tikzcd}
\]
And let $D_\Delta\coloneqq D_{X^\Delta}$ be the set of lattice points of the independence polytope of (the polymatroid rank function $\varphi_{X^\Delta}$ of) $X^\Delta$; see \autoref{subsection:polytopes-vector-spaces}.

By \cite[Thm.~6.4, p.~22]{Esteves2025}, the irreducible components of
$\BLP(\mathfrak V)$ are
$\BLP(\mathfrak V)_v\coloneqq X^{\{v\}}$, for $v\in H$. For each polygon $\Delta\subseteq H$, put
\[
    \mathbb{LP}(\mathfrak{V})_\Delta\coloneqq
    \bigcap_{v\in\Delta}\BLP(\mathfrak V)_v.
\]

\begin{prop}\label{prop:polygon-intersection-hilbert-polynomial}
Let $\Delta\subseteq H$ be a polygon. Then $\mathbb{LP}(\mathfrak{V})_\Delta=X^\Delta$. In particular, $\mathbb{LP}(\mathfrak{V})_\Delta$ is a multiplicity-free subvariety of $\BP_H$.
\end{prop}

\begin{proof} Since $\BP(W^\Delta)$ is integral, so is $X^\Delta$. On the other hand, since $\BLP(\mathfrak V)$ is normal-crossings by 
\autoref{thm:linked-projective-space-snc}, every
scheme-theoretic intersection of its irreducible components, and in
particular $\mathbb{LP}(\mathfrak{V})_\Delta$, is reduced. Finally, it follows from \cite[Prop.~6.2, p.~20]{Esteves2025} that $\mathbb{LP}(\mathfrak{V})_\Delta$ and $X^\Delta$ coincide set-theoretically. Hence, $\mathbb{LP}(\mathfrak{V})_\Delta=X^\Delta$.

The second assertion follows from the first and \cite[Thm.~1.1, p.~4191]{Li_2018}. 
\end{proof}

We need two lemmas before stating the main result of this article. 

\begin{lemma}\label{lem:intersection-polygon}
Let $v_1,\dots,v_m$ forming an oriented $m$-gon $\Delta\subseteq H$. For each $i\in[m]$, let $\pi_i \coloneqq R^{\Delta}_{v_i}\cap H$, the shadow region in $H$ of $v_i$ in $\Delta$. Put $F_j\coloneqq \pi_1\cup\cdots\cup \pi_j$ and $F_j^c\coloneqq H - F_j$ for each $j\in[m]$. Then
\[
    \bP_\Delta=\bP_{v_1}\cap \left\lbrace q\in\BR_{\geq0}^H\,\middle|\, q(F_j^c)=\rho_{v_1}(F_j^c) ~ \forall j\in[m-1]\right\rbrace.
\]
In addition, we have \[\bP_{\Delta}=\bigcap_{v\in \Delta} \bP_v.\] 
\end{lemma}

\begin{proof}
The first statement follows from
\cite[Lem.~4.4(2)]{Chow.class} and \cite[Prop.~2.5]{amini2024tropicalizationlinearseriestilings}. 

As for the second statement, assume first that $m=2$. By \cite[Lem.~5.1]{Chow.class} and \cite[Prop.~3.4(I)(3)]{amini2024tropicalizationlinearseriestilings}, the relative interiors of  $\bP_{v_1}$ and $\bP_{v_2}$ are disjoint. Since the polytopes $\bP_v$ form a tiling by \cite[Thm.~5.2]{Chow.class}, the intersection $\bF\coloneqq \bP_{v_1}\cap \bP_{v_2}$ is a proper common face of  $\bP_{v_1}$ and $\bP_{v_2}$. On the other hand, by \cite[Lem.~4.4(1)(2)]{Chow.class}, $\bP_\Delta$ is a common facet of $\bP_{v_1}$ and $\bP_{v_2}$. Since $\bP_\Delta\subseteq\bF$ and  $\bF$ is proper, we conclude that $\bF=\bP_\Delta$. 

We consider now the general case. 
Since $\Delta$ is a polygon, so is $\Gamma_j=\lbrace v_1,v_{j+1}\rbrace$ for each $j\in[m-1]$. Also, $R^{\Gamma_j}_{v_1}\cap H=F_j$ and  $R^{\Gamma_j}_{v_{j+1}}\cap H=F_j^c$. Thus, 
by the first statement, 
\[  
    \bP_{v_1}\cap \left\lbrace q\in\BR_{\geq0}^H\,\middle|\,q(F^c_j)=\rho_{v_1}(F_j^c)\right\rbrace = \bP_{\Gamma_j}. 
\] 
Hence, using the first statement and the second for $m=2$, we have
\[
\bP_\Delta=\bigcap_{j=1}^{m-1}  \bP_{\Gamma_j} =\bigcap_{j=1}^{m-1}  \bP_{v_1}\cap \bP_{v_{j+1}} =\bigcap_{j=1}^{m} \bP_{v_j},
\]
as required.
\end{proof}
For each nonempty polytope \(\bP\subseteq\bR^H\), denote by \(\operatorname{dir}(\bP)\) the linear space parallel to its affine span, equivalently
\[
\operatorname{dir}(\bP)\coloneqq \operatorname{span}\{x-y\mid x,y\in \bP\}.
\]
\begin{lemma}\label{P_Delta-gives-Delta}
    Put $\Omega_R\coloneqq\{x\in\BR_{\geq0}^H\mid x(H)=R\}$, where $R\coloneqq r+1$.
    The following statements hold: \begin{enumerate}
        \item If $\Delta,\Delta'\subseteq H$ are polygons, and $\bP_\Delta=\bP_{\Delta'}$, then $\Delta=\Delta'$.
        \item If $\Sigma\subseteq H$ is nonempty subset such that
        \[
            \Big(\bigcap_{v\in\Sigma}\bP_v \Big)\cap~ \operatorname{relint}(\Omega_R)\neq\varnothing
        \]
        then $\Sigma$ is a polygon.
    \end{enumerate}
\end{lemma}
\begin{proof} Let $v_1,\dots,v_m$ forming an oriented polygon $\Delta$. Let $\pi_i\coloneqq R^\Delta_{v_i}\cap H$ and $F_i\coloneqq\pi_1\cup\cdots\cup\pi_i$ for each $i\in[m]$, and put $\pi\coloneqq(\pi_1,\dots,\pi_m)$. Then by \autoref{lem:intersection-polygon}, we have
\[
    \bP_\Delta=\bP_{v_1}\cap\{q\in\BR_{\geq0}^H\,|\,q(F_j^c)=\rho_{v_1}(F_j^c)\,\forall j\in [m-1]\}.
\]
Moreover, by \cite[Lem.~4.4]{Chow.class}, the polytope $\bP_\Delta$ has codimension $m-1$ in $\Omega_R$. Hence the affine span of  $\bP_\Delta$ is obtained from the affine span of $\Omega_R$ by imposing the equations $q(F_j^c)=0$ with $j\in[m-1]$. Equivalently, 
\[
    \operatorname{dir}(\bP_\Delta)=\left\{q\in\BR^H\,|\,q(\pi_j)=0\,\forall j\in [m]\right\}.
\]

Let now $\Delta'$ be another polygon such that $\bP_\Delta=\bP_{\Delta'}$, and let us proof that $\Delta'=\Delta$. First, $\operatorname{dir}(\bP_\Delta)=\operatorname{dir}(\bP_{\Delta'})$, and thus the unordered partitions induced by $\Delta$ and $\Delta'$ are equal. For each $i\in[m]$, let $u_i$ be the vertex in $\Delta'\cap\pi_i$, and put $S_i\coloneqq W_{v_i}\cap V_{\pi_i}$ and $S'_i\coloneqq W_{u_i}\cap V_{\pi_i}$. Then 
\[
\prod_{i=1}^m\bP_{S_i}=\bP_{\Delta}=\bP_{\Delta'}=\prod_{i=1}^m\bP_{S'_i}
\]
and thus $\bP_{S_i}=\bP_{S'_i}$ for each $i\in[m]$. Furthermore, since $\bP_{\Delta}$ has codimension $m-1$ in $\Omega_R$, we must have that $S_i$ and $S'_i$ are simple for each $i\in[m]$, that is, their polytopes have dimension $|\pi_i|-1$; see \cite[Sec.~2.5]{Chow.class}

Now, by \cite[Lem.~5.1]{Chow.class}, for each $i\in[m]$, there exists a nonzero $c_i\in\bk^H$ such that $c_iW_{v_i}\subseteq W_{u_i}$. Furthermore, by the proof of loc.~cit, we may choose $c_i$ such that $c_i(u_i)\neq 0$, and $c_i(v_i)=0$ if $v_i\neq u_i$. Now, $c_i|_{\pi_i}S_i\subseteq S'_i$. By \cite[Prop.~2.2]{Chow.class}, since $\bP_{S}=\bP_{S'}$, and $c_i(u_i)\neq 0$, we must have that $c_i|_{\pi_i}$ is invertible, whence $c_i(v_i)\neq 0$, and thus $v_i=u_i$. So $\Delta'=\Delta$.

Consider now a subset $\Sigma\subseteq H$ as in the second statement, and put $\bF\coloneqq\bigcap_{v\in\Sigma}\bP_v$. By \cite[Thm.~5.2]{Chow.class} the $\bP_v$ form a polyhedral tiling. Then $\bF$ is a face of $\bP_v$ for each $v\in \Sigma$. Thus, by \cite[Prop.~2.7]{amini2024tropicalizationlinearseriestilings}, for each $v\in V$, there is a partition $\pi_v=(\pi_{v,1},\dots,\pi_{v,m})$ of $H$  such that 
\[
\bF=\bP_v\cap \lbrace q\in\BR_{\geq0}^H\,|\,q(F_j^c)=\rho_v(F_j^c) \,\forall j\in [m-1]\rbrace,
\]
where $F_j\coloneqq\pi_{v,1}\cup\cdots\cup\pi_{v,j}$ for each $j\in[m]$. Since $\bF\cap\operatorname{relint}(\Omega_R)\neq\varnothing$, by \cite[Prop.~4.5]{Chow.class} there is a $m$-gon $\Delta_v$ such that $v\in\Delta_v$ and $\bP_{\Delta_v}=\bF$. In particular, $\bP_{\Delta_v}=\bP_{\Delta_w}$, and thus, by the first statement, $\Delta_v=\Delta_w$ for each $v,w\in\Sigma$. Let $\Delta$ be the common polygon. By construction, for each $v\in\Sigma$, we have $v\in \Delta_v=\Delta$. Then $\Sigma\subseteq\Delta$. Since $\Delta$ is a polygon, so is $\Sigma$. This finishes the proof of the lemma.
\end{proof}

\begin{thm}\label{thm:multivariate-hilbert-polynomial}
Let $\mathfrak{V}$ be a nontrivial exact finitely generated linked net of
$\bk$-vector spaces of dimension $r+1$ over a $\BZ^n$-quiver $Q$, and let $H$ be
its minimum set of generators. Then
\[
    P_{\BLP(\mathfrak V)}\bigl(\Bn\bigr)=\binom{r+\sum_{u\in H}n_u}{r}.
\]
In words, $\BLP(\mathfrak V)$ has the multivariate Hilbert polynomial of the small
diagonal $\BP^r\hookrightarrow(\BP^r)^H$.
\end{thm}

\begin{proof} The hypotheses of
\autoref{prop:normal-crossings-euler-cancellation} are verified for $Y\coloneqq\BLP(\mathfrak V)$. Indeed, $\BLP(\mathfrak V)$ is simple normal-crossings by
\autoref{thm:linked-projective-space-snc}, and its nonempty strata are the
$Y_\Delta\coloneqq \BLP(\mathfrak V)_\Delta$ for polygons $\Delta\subseteq H$, hence multiplicity-free varieties by \autoref{prop:polygon-intersection-hilbert-polynomial}.

For each $a\in\BZ_{\geq0}^H$, let $K_a$ be the simplicial complex on
$E\coloneqq E_Y$ whose nonempty faces are the nonempty $\Delta\subseteq E$ such that $Y_\Delta\neq\varnothing$ and $a\in D_{Y_\Delta}$. By  \cite[Prop.~6.3, p.~21]{Esteves2025}, identifying $E_Y$ with $H$, we have that $Y_\Delta\neq\varnothing$ if and only if $\Delta$ is a polygon in $H$. Also, $D_{Y_\Delta}=D_\Delta$ by \autoref{prop:polygon-intersection-hilbert-polynomial}. 

If $a(H)>r$, then
\autoref{prop:general-vector-space-hilbert-polynomial} yields that $a\notin D_\Delta$ for any polygon $\Delta\subseteq H$, whence $K_a$ has only the empty face. 

Suppose $a(H)\leq r$. We prove that $K_a$ is contractible, whence connected and acyclic.
Choose
$0<\varepsilon<1/|H|$ and set
\[
    C_{a,\varepsilon}
    \coloneqq a+\varepsilon\mathbf1
    +\Omega_{R-a(H)-\varepsilon|H|},
\]
where 
\[
    \Omega_t=\{x\in\BR_{\geq0}^H\mid x(H)=t\} \quad \text{for each } t\ge0.
\]
Note that $C_{a,\varepsilon}$ is a nonempty closed convex subset of
$\operatorname{relint}(\Omega_R)$. 

For each $v\in H$, set $U_{v,a}\coloneqq
\bP_v\cap C_{a,\varepsilon}$. By
\cite[Thm.~5.2]{Chow.class}, we have
 \[
\Omega_R=\bigcup_{v\in H}\bP_v.
\] 
Hence, the sets $U_{v,a}$ cover
$C_{a,\varepsilon}$. For each nonempty subset $\Delta\subseteq H$, we have
\begin{equation}\label{eq:cover-intersections}
    \bigcap_{v\in\Delta}U_{v,a}
    =\bigg(\bigcap_{v\in\Delta}\bP_v\bigg)\cap C_{a,\varepsilon}.
\end{equation}
If the intersection in \eqref{eq:cover-intersections} is nonempty, then it meets $\operatorname{relint}(\Omega_R)$. Then $\Delta$ is a polygon by 
\autoref{P_Delta-gives-Delta}, and 
\[ 
\bigcap_{v\in\Delta}\bP_v=\bP_\Delta,
\] 
by \autoref{lem:intersection-polygon}. Using \autoref{prop:interior-criterion} now, we get that the right-hand side of
\eqref{eq:cover-intersections} is nonempty exactly when $\Delta$ is a polygon and 
$a\in D_\Delta$. Thus the nerve of the cover
$(U_{v,a})_{v\in H}$ is precisely $K_a$.

The $U_{v,a}$ are closed convex sets, as are all their nonempty finite
intersections. By Nerve Theorem for finite closed convex covers
\cite[Thm.~3.1, p.~17]{BauerKerberRollRolle_2023},  we have that
$K_a$ is homotopically equivalent to $ C_{a,\varepsilon}$. Since $C_{a,\varepsilon}$ is nonempty and
convex, $K_a$ is contractible. 

Consequently, $K_a$ has a nonempty face exactly when $a(H)\leq r$, in which case $K_a$ is acyclic. We may now apply
\autoref{prop:normal-crossings-euler-cancellation}, obtaining \[
 P_X(\Bn)=\sum_{\substack{a\in \bigcup_{v\in H} D_{Y_v}}}
 B_a(\Bn).
\]
Then, by \cite[Thm 5.2]{Chow.class}, we have \[
P_X(\Bn)=\sum_{\substack{a\in\BZ_{\geq0}^H\\a(H)\leq r}}
B_a(\Bn).
\] 

Put $N\coloneqq\sum_{u\in H}n_u$. From Multivariate Chu--Vandermonde identity \cite[Thm.~5.2, p.~63]{Vignat_Moll_2015}, it follows that
\[
\sum_{\substack{a\in\BZ_{\geq0}^H\\a(H)=s}} B_a(\Bn)  =\binom{N+s-1}{s}.
\] 
Summing over $0\leq s\leq r$ and using the hockey-stick identity gives
\[
    P_X(\Bn)=\sum_{s=0}^r\binom{N+s-1}{s} =\binom{N+r}{r},
\]
finishing the proof. 
\end{proof}

\bibliographystyle{amsalpha}
\bibliography{ref}

\end{document}